\documentclass[11pts]{article}
 
\usepackage{amssymb,amsthm}
\usepackage{amsmath}
\usepackage{latexsym}
\usepackage{amssymb}
\usepackage[dvips]{graphics}
\usepackage{graphicx}
\newtheorem{theorem}{Theorem}

\newtheorem{lemma}[theorem]{Lemma}

\newtheorem{definition}[theorem]{Definition}

\newtheorem{example}[theorem]{Example}

\newcommand{\R}{{\mathbf R}}
\newcommand{\Z}{{\mathbf Z}}

\newcommand{\Q}{\mathbf Q}

\newcommand{\E}{{\mathbb E}}
\renewcommand{\P}{\mathbb P}

\newcommand{\LL}{\mathcal L}

\newcommand{\cd}{\rm {cd}}
\newcommand{\st}{\rm {St}}

\begin{document}

\title{The asphericity of random 2-dimensional complexes}         
\author{A.E. Costa and M. Farber}        
\date{October 22, 2012}          
\maketitle
\abstract{ We study random 2-dimensional complexes in the Linial - Meshulam model and prove that for the probability parameter satisfying $$p\ll n^{-46/47}$$ a random 2-complex
 $Y$ contains several pairwise disjoint tetrahedra such that the 2-complex $Z$ obtained by removing any face from each of these tetrahedra is aspherical. Moreover, we prove that the obtained complex $Z$ satisfies the Whitehead conjecture, i.e. any subcomplex $Z'\subset Z$ is aspherical. 
This implies that $Y$ is homotopy equivalent to a wedge $Z\vee S^2\vee \dots \vee S^2$ where $Z$ is a 2-dimensional aspherical simplicial complex. 
We also show that under the assumptions $$c/n<p<n^{-1+\epsilon},$$ where $c>3$ and $0<\epsilon<1/47$, the complex 
$Z$ is genuinely 2-dimensional and in particular, it has sizable 2-dimensional homology; it follows that in the indicated range of the probability parameter $p$ the cohomological dimension of the fundamental group $\pi_1(Y)$ of a random 2-complex equals 2.
}

\section{Introduction}

The problem of modeling of large systems motivates the development of mixed probabilistic - topological concepts, 
including high-dimensional generalisations of the Erd\"os and R\'enyi random graphs of  \cite{ER}.
A probabilstic model of this type was recently suggested and studied by Linial-Meshulam ~\cite{LM} and 
Meshulam-Wallach ~\cite{MW}. In this model one generates a random $d$-dimensional complex $Y$ by considering the full $d$-dimensional skeleton of the simplex 
$\Delta_n$ on vertices $\{1, \dots, n\}$ and retaining $d$-dimensional faces independently with probability $p$.  
The work of Linial--Meshulam and Meshulam--Wallach provides threshold functions for the vanishing of the $(d-1)$-st homology groups of random complexes with coefficients in a finite abelian group. Threshold functions for the vanishing of the $d$-th homology groups were subsequently studied by Kozlov \cite{Ko}. 


%

The fundamental group of a random 2-complexes $Y$ was investigated by Babson, Hoffman, and Kahle \cite{BHK}.  
They showed that the fundamental group of a random 2-complex is either Gromov hyperbolic or trivial, except when the probability parameter satisfies 
$p\sim n^{-1/2}$. 

In the paper \cite{CCFK} it was proven that a random 2-complex $Y$ collapses to a graph if $p\ll n^{-1}$ and therefore its fundamental group is free. 

The preprint \cite{ALLM} suggests an explicit constant $\gamma$ such that for any $c<\gamma$ a random $Y\in Y(n, c/n)$ is either collapsible to a graph or it contains a  tetrahedron.

For several reasons it is important to have a model producing random {\it aspherical} 2-dimensional complexes $Y$. 
A connected simplicial complex $Y$ is said to be aspherical if $\pi_i(Y)=0$ for all $i\ge 2$; this is equivalent to the requirement that the universal cover of $Y$ is contractible. 
It is well-known that for 2-dimensional complexes $Y$ the asphericity is equivalent to the vanishing of the second homotopy group $\pi_2(Y)=0$, or equivalently, any continuous map 
$S^2\to Y$ is homotopic to a constant map. 

Random aspherical 2-complexes could be helpful for testing probabilistically the open problems of two-dimensional topology, such as the Whitehead conjecture. 
This conjecture stated by J.H.C. Whitehead in 1941 claims that a subcomplex of an aspherical 2-complex is also aspherical. 
Surveys of results related to the Whitehead conjecture can be found in \cite{B}, \cite{R}. 

Unfortunately, in the Linial-Meshulam model a random 2-complex is aspherical only if $p\ll n^{-1}$, when it is homotopy equivalent to a graph; however for 
$p\gg n^{-1}$ a random 2-complex is never aspherical since it contains a tetrahedron as a subcomplex. 
For this reason in this paper we modify the notion of asphericity and speak about {\it asphericable} complexes, which can be made aspherical by deleting faces without affecting their fundamental groups (see the precise definition below). 
Our main result Theorem \ref{thm1} states that random 2-complexes in the Linial - Meshulam model are asphericable in the range $p\ll n^{-46/47}$. This implies, in particular, that 
for $c>3$ and $0<\epsilon <1/47$ 
the cohomological dimension of the fundamental group of a random 2-complex equals 2 assuming that $c/n <p<n^{-1+\epsilon}$. 
We refer the reader to \cite{Br} for the definition of the notion of cohomological dimension. 

The method of the proof of Theorem \ref{thm1} is combinatorial and is based on the observation that a large proportion of vertices of any triangulation of a surface have small degree (Lemma \ref{lm1}); moreover, one may find a finite set of local structures which are present in any simplicial sphere (Theorem \ref{comb}). This allows
finding a finite list of 2-complexes which are present in any regular quotient of a simplicial 2-sphere, see Theorem \ref{thmlist}. Thus the complexes which have no 
fragments from this finite list are asphericable; the proof of the main result employs the technique of containment problem for random 2-complexes.

\section{The main results} 

 In the paper, we use standard terminology regarding simplicial complexes, see \cite{Hat02}; 1-simplexes are also called edges; 2-simplexes are also referred to as faces. A triangulation of a surface is a simplicial complex homeomorphic to the surface. 

\begin{definition}\label{def1} {\rm A finite simplicial 2-dimensional complex $Y$ will be called} {\it asphericable} {\rm if $Y$ contains a set of pairwise face disjoint tetrahedra 
$T_1, \dots, T_k\subset Y$ such that any subcomplex $Z\subset Y$ with the property that $T_i\not\subset Z$ for all $i=1, \dots, k$ is aspherical.} 
\end{definition} 

By a {\it tetrahedron} we mean a complex isomorphic to the boundary of the 3-simplex, i.e. a simplicial complex with four vertices and four faces. 
Two subcomplexes of a complex are called {\it face disjoint} if they have no common faces. 

Any subcomplex of an asphericable complex is also asphericable. 

Note that asphericable complexes without tetrahedra are exactly the complexes satisfying the Whitehead conjecture, i.e. they are 2-complexes which are aspherical and such that all their subcomplexes are aspherical.  

If $Y$ is asphericable then the subcomplex $Z\subset Y$ obtained from $Y$ by removing a face from each of the tetrahedra $T_1, \dots, T_k$ is aspherical.\footnote{This explains our term \lq\lq asphericable\rq\rq\, which intends to mean \lq\lq can be made aspherical\rq\rq.} 
It is clear that in this case $Y$ has the homotopy type of the wedge 
\begin{eqnarray}\label{wedge}
Z\vee S^2\vee \cdots\vee S^2\end{eqnarray} of $Z$ and $k$ spheres $S^2$. 
This implies that for an asphericable 2-complex $Y$ the fundamental group $\pi_1(Y)$ has cohomological dimension less or equal than $2$. Indeed, $\pi_1(Y)=\pi_1(Z)$ and $Z$ is aspherical and 2-dimensional and therefore the chain complex of 
the universal cover of $Z$ is a free resolution of the infinite cyclic group $\Z$ over the group ring $\Z[\pi_1(Y)]$ having length 2; therefore the cohomological dimension of $\pi_1(Y)$ is $\le 2$. From the wedge decomposition (\ref{wedge}) one also sees that the universal cover of $Y$ is equivariantly homotopy equivalent to a wedge of 2-spheres and the fundamental group acts on these spheres by permutations; hence we obtain that the second homotopy group
$\pi_2(Y)$ is free as module over the group ring of $\pi_1(Y)$.

Consider the Linial-Meshulam model $Y(n,p)$ of random 2-complexes \cite{LM}. Recall that $Y(n,p)$ is the probability space of all 2-complexes
$Y$ on the vertex set $\{1, \dots, n\}$ containing the full 1-skeleton, i.e. 
$$\Delta_n^{(1)}\subset Y \subset \Delta_n^{(2)}$$
(where $\Delta_n$ is the simplex with vertices $\{1, \dots, n\}$) and the probability function $\P: Y(n,p)\to \R$ is given by the formula
$$\P(Y)= p^{f_2(Y)}(1-p)^{{\binom n 3}-f_2(Y)}$$
where $f_2(Y)$ denotes the number of 2-simplexes in $Y$.

From \cite{CCFK} we know that for $p\ll n^{-1}$ a random 2-complex $Y\in Y(n, p)$ collapses simplicially to a graph, a.a.s. This 
implies that for $p\ll n^{-1}$ a random 2-complex is aspherical and moreover
the fundamental group $\pi_1(Y)$ is free, a.a.s. 

The main result of this paper states:

\begin{theorem}\label{thm1} Assume that \begin{eqnarray}\label{one} p \ll n^{-\frac{46}{47}}.\end{eqnarray} 
Then a random 2-complex $Y\in Y(n,p)$ is asphericable, a.a.s.\footnote{The symbol a.a.s. is an abbreviation of \lq\lq asymptotically almost surely\rq\rq, which means that the probability that the corresponding statement is valid tends to 1 as $n\to \infty$.}
\end{theorem}     

Theorem \ref{thm1} implies that under the assumptions (\ref{one}) the fundamental group $\pi_1(Y)$ of a random 2-complex has cohomological dimension $\le 2$ and $\pi_2(Y)$ is free as 
$\Z[\pi_1(Y)]$-module, a.a.s.

A random 2-complex $Y$ is not asphericable for $p\gg n^{-5/6}$, a.a.s. 
Indeed, in this range a random 2-complex contains, as a simplicial subcomplex, a triangulation $K$ of the sphere 
$S^2$ with 5 vertices and 6 faces (as follows from Theorem \ref{embed}). This subcomplex $K\subset Y$ does not contain a tetrahedron and is not aspherical, and hence $Y$ is not asphericable.

\begin{theorem}\label{thm2}
(A) If for some constants $3< c$ and $0< \epsilon < 1/47$
the probability parameter $p$ satisfies
\begin{eqnarray}\label{two}
\frac{c}{n} < p < n^{-1+\epsilon} 
\end{eqnarray}
then for the second Betti number of the fundamental group\footnote{The Betti numbers of a discrete group $G$ are defined as the Betti numbers of an aspherical CW complex with fundamental group $G$, see \cite{Hat02}, \cite{Br}.}
$\pi_1(Y)$, where $Y \in Y(n,p)$, one has 
\begin{eqnarray}\label{secondbetti}
n^2\cdot \frac{c-3}{8}\,\le \, b_2(\pi_1(Y)) \, \le\,  n^{2+\epsilon},
\end{eqnarray}
a.a.s.

(B) In range (\ref{two})
the cohomological dimension of the fundamental group $\pi_1(Y)$ of a random 2-complex $Y\in Y(n, p)$ equals $2$, a.a.s. 
\end{theorem}

See \S \ref{prfthm2} for the proof. 

There exists a triangulation $K$ of the real projective plane ${\bf {RP}}^2$ having 6 vertices and 10 faces and for $p\gg n^{-3/5}$ the complex $K$ embeds into a random 2-complex $Y$, a.a.s. We expect that the induced homomorphism $\pi_1(K) \to \pi_1(Y)$ is injective a.a.s., which would imply that $\pi_1(Y)$ has elements of order two and hence its cohomological dimension is infinite. 

\section{A combinatorial reduction of Theorem \ref{thm1}}

In this section we explain how the main Theorem \ref{thm1} follows from a deterministic combinatorial result 
which is stated in this section as Theorem \ref{thmlist} and plays a key role in the proof of Theorem \ref{thm1}.
We first recall some notations and results from \cite{CCFK}; they are similar to the classical results on the containment problem for random graphs \cite{JLR}.

\begin{definition}\label{def6} {\rm For a simplicial 2-complex $S$ let 
$\mu(S)$
 denote 
$\mu(S)  = 
v/f 
\in \, \Q,$
where $v=v_S$ and $f=f_S$ are the numbers of vertices and faces in $S$. 
Define also
\begin{eqnarray}
\tilde \mu(S) = \min_{S'\subset S} \mu(S'),
\end{eqnarray}
where the minimum is formed over all subcomplexes $S'\subset S$ or, equivalently, over all pure\footnote{Recall that a 2-dimensional simplicial complex is called {\it pure} if it is the union of 2-dimensional simplexes.}
subcomplexes $S' \subset S$. }
\end{definition}

The following Theorem explains importance of the invariant $\tilde \mu(S)$ for the containment problem of stochastic topology.

\begin{theorem}[{\rm \cite{BHK}, \cite{CCFK}}]\label{embed} Let $S$ be a finite simplicial 2-dimensional complex.  
\begin{enumerate}
  \item[(A)] If $p\ll n^{-\tilde \mu (S)}$ then the probability that $S$ admits a  simplicial embedding into a random 2-complex $Y\subset Y(n, p)$ tends to zero as $n\to \infty$.
  \item[(B)] If $p\gg n^{-\tilde \mu (S)}$ then the probability that $S$ admits a simplicial embedding into a random 2-complex $Y\subset Y(n, p)$ tends to one as $n\to \infty$. 
\end{enumerate}
\end{theorem}

\begin{example} {\rm Consider a simplicial graph $\Gamma$ and the cone over it $S=C(\Gamma)$. 
One has $v_S=v_\Gamma+1$ and $f_S=e_\Gamma$. Therefore 
$\mu(S) = \frac{v_\Gamma+1}{e_\Gamma}.$
Thus, $\mu(S)\le 1$ if $\chi(\Gamma)< 0$. }
\end{example}

\begin{example} \label{ex3points} {\rm 
As another example consider a simplicial 2-complex $S$ homeomorphic to the 2-disc. It is easy to see that $f=v+v_i-2$ and hence
\begin{eqnarray}\label{disc}\mu(S) = \frac{v}{v+v_i-2},\end{eqnarray}
where $v_i$ denotes the number of internal vertices of $S$. This formula is important for the method of the paper; in particular, it shows that for disc triangulations one has $\mu(S) <1$ assuming that 
$v_i\ge 3$. In the proof of Theorem \ref{thm1} we will show the existence of  a sub-disc in any triangulated sphere having at most 46 vertices, among them 3 are internal; for such a subdisc one has $\mu\le 46/47$ according to formula (\ref{disc}).  }
\end{example}

\begin{example}{\rm 

\begin{figure}[h]
\centering
\includegraphics[width=0.8\textwidth]{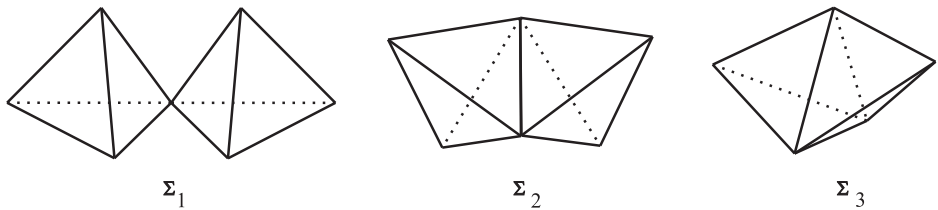}
\caption{Complexes $\Sigma_1$, $\Sigma_2$, $\Sigma_3$.}\label{lone1}
\end{figure}
Figure \ref{lone1} shows complexes $\Sigma_i$ which are unions of two tetrahedra having $i$ vertices in common, where $i=1,2,3$. 
One has $\mu(\Sigma_1) = 7/8$, $\mu(\Sigma_2)= 3/4$ and $\mu(\Sigma_3) = 5/7$. Since $\mu(\Sigma_i) \le 46/47$, using Theorem \ref{embed} one sees that for 
$p\ll n^{-46/47}$ a random 2-complex contains no subcomplexes isomorphic to $\Sigma_i$. 
Thus, according to Theorem \ref{thm1} and Theorem \ref{embed}, the picture of a random 2-complex with $p\ll n^{-46/47}$ is as follows: it contains a family of pairwise disjoint (and not only face disjoint) tetrahedra such that removing a face from each of these tetrahedra gives an aspherical 2-complex satisfying the Whitehead conjecture. 
}\end{example}

Next we mention the following useful formula which will be used later in this paper
\begin{eqnarray}\label{muel}
\mu(S) = \frac{1}{2} + \frac{2\chi(S)+ L(S)}{2f},
\end{eqnarray}
where $S$ is a finite simplicial 2-complex, $f=f_S$ is the number of faces in $S$ and 
\begin{eqnarray}\label{mu1}
L(S) = \sum_{e}\left(2-\deg(e)\right);
\end{eqnarray} 
here $e$ runs over all edges of $S$ and $\deg(e)$ denotes the number of faces incident to $e$. Formula (\ref{muel}) follows directly from definitions using the equation $\sum_e \deg(e) = 3f$.


\begin{definition}\label{reg} {\rm Let $M$ be a simplicially triangulated closed surface. Let $f: M\to M'$ be a simplicial surjective map. We will say that $M'$ is a} {\it regular simplicial quotient} {\rm of $M$ if (1) for any simplex $\sigma\in M$ the image $f(\sigma)$ has the same dimension, i.e. $\dim f(\sigma) = \dim \sigma$, and (2) if for two distinct 
2-dimensional simplexes $\sigma_1$ and $\sigma_2$ of $M$ one has $f(\sigma_1)=f(\sigma_2)$ then $\dim (\sigma_1\cap \sigma_2) <1$.  }
\end{definition}

We will use the following Lemma which was explicitly stated by Mark A. Ronan:

\begin{lemma}  [\cite{Ron}, Lemma 2.1]\label{ronan}
Any non-aspherical 2-dimensional complex $Y$ contains as a subcomplex a regu\-lar simplicial quo\-tient of a simplicial 2-sphere. 
\end{lemma}

The method of the proof of Theorem \ref{thm1} is to study in detail the local structure of the regular quotients of spheres; in particular we establish the following 
(deterministic) result:

\begin{theorem}\label{thmlist}\label{cor1} There exists a finite list $\LL$ of compact 2-dimensional simplicial complexes with the following two properties: 

(1) a finite simplicial 2-complex $Y$ is asphericable if it contains no subcomplexes isomorphic to a complex from the list $\LL$.

(2) for any $S\in\LL$ one has 
\begin{eqnarray}\label{46/47}\tilde \mu(S)\le \frac{46}{47};\end{eqnarray}
\end{theorem}

The proof is given in section \S \ref{prfthmlist} below. The list $\LL$ will be defined as the union 
\begin{eqnarray}\label{union} \LL = \LL_1 \cup \LL_2
\end{eqnarray} with $\LL_1$ consisting of the single 2-complex
$\LL_1=\{\Sigma\},$ where  $\Sigma$ is the union of two tetrahedra having $3$ vertices in common. 
One has $\mu(\Sigma) = 5/7$, i.e. $\Sigma$ satisfies (\ref{46/47}). 

%

\begin{proof}[Proof of Theorem \ref{thm1}]
Theorem \ref{thm1} directly follows from Theorem \ref{cor1} and Theorem \ref{embed}.  
\end{proof}

A proof of Theorem \ref{thmlist} will be given later in section \ref{prfthmlist}. The following section contains an auxiliary combinatorial material used in the proof of Theorem \ref{thmlist}.

\section{Local structure of simplicial spheres}

Let $M$ be a closed (i.e. compact without boundary) simplicially triangulated surface. Denote by $v, e, f$ the numbers of vertices, edges and faces of $M$. 
One has, using the Euler-Poincar\'e formula, 
$$\sum_x \deg(x) = 2e = 3f= 6v-6\chi(M)$$ where $x$ runs over the vertices of $M$. 
Here $\deg(x)$ denotes the degree of a vertex $x$ in the 1-skeleton, i.e. the number of edges incident to $x$. 
Thus
\begin{eqnarray}\label{average1}\frac{\sum_x \deg(x)}{v} = 6 - \frac{6 \chi(M)}{v},\end{eqnarray}
which gives a formula for the average degree of a vertex. We see that if the topological type of the surface is fixed and the number of vertices $v\to \infty$ in the triangulation tends to infinity, then the average degree of a vertex tends to $6$. This statement is strengthened in the following Lemma giving lower bounds on the number of vertices of lower degree. 

\begin{lemma} \label{lm1} Let $M$ be a closed simplicially triangulated surface. Let $k\ge 6$ be an integer and let $l_k$ denote the number of vertices of $M$ having degree $\le k$. Then 
\begin{eqnarray}\label{ge1}
l_k \ge \frac{(k-5)v+6\chi(M)}{k-2}. 
\end{eqnarray}
In particular, if $M$ is a simplicial sphere, real projective space, torus or Klein bottle then $\chi(M) \ge 0$ and one has
\begin{eqnarray}\label{ge2}
l_k \ge \frac{k-5}{k-2}v.
\end{eqnarray}
\end{lemma}
\begin{proof}
Denote by $V$ the set of vertices of $M$. Let ${\frak L}_k\subset V$ denote the set of vertices $x$ with $\deg(x) \le k$. One has
$$\sum_{x\in V} \deg(x) = \sum_{x\in {\frak L}_k} \deg(x) + \sum_{x\in V-{\frak L}_k}\deg(x) \ge 3l_k + (k+1)(v-l_k).$$
Combining this inequality with (\ref{average1}) gives $6v-6\chi(M) \ge 3l_k + (k+1)(v-l_k)$ which is equivalent to (\ref{ge1}). 
\end{proof}

As an example, we mention that for $k=8$ and $\chi(M)\ge 0$ one obtains an inequality 
$l_8\ge \frac{v}{2}, $ i.e. at least half of the vertices have degree $\le 8$. For $k=17$ one obtains $$l_{17} \ge \frac{4}{5}v$$ (again, assuming that $\chi(M)\ge 0$).


\begin{theorem}\label{comb} Let $M$ be a simplicial triangulation of the sphere $S^2$. Then there exist three vertexes 
$x, y, z$ of $M$ satisfying $\deg(x)\le 17, \, \deg(y) \le 17, \, \deg(z) \le 17,$
such that either (a) the points $x, y, z$ span a 2-simplex, or (b) for some vertex $w$ the triples 
$x, y, w$ and $y, z, w$ span 2-simplexes of $M$. 
\end{theorem}
\begin{proof} We shall assume that $M$ has at least $6$ vertexes, since otherwise our statement is obvious. 

Let ${\frak L}_{17}$ denote the set of vertexes of $M$ of degree $\le 17$. The cardinality $l_{17}=|{\frak L}_{17}|$ of this set satisfies 
$l_{17}\ge 4/5\cdot v$, see above. Let $L$ denote the graph with vertex set $\frak L_{17}$ induced from the 1-skeleton of $M$. In general, $L$ has many connected components; for any $i=1, 2, \dots, $ we will denote by $L_i$ the union of connected components of $L$ having exactly $i$ vertexes. 
We will also denote 
$$S_i=\cup_{x\in V(L_i)}\st(x),$$ the union of 2-simplexes having at least one vertex in $L_i$. 

Note that $L_1$ is the union of isolated vertexes and $S_1$ in the union of stars of these vertexes;
moreover, for $x, y \in V(L_1)$ the stars $\st(x)$ and $\st(y)$ have no common faces. Since each of these stars has at least 3 faces, we obtain 
\begin{eqnarray}\label{fsone} f(S_1) \ge 3v(L_1).
\end{eqnarray}

The graph $L_2$ is the union of isolated edges and $S_2$ is the union of stars of these edges. If $e, e'\subset L_2$ are two distinct edges then the stars 
$\st(e)$ and $\st(e')$ are face disjoint, i.e.have no common faces. Clearly, each star $\st(e)$ contains $\ge 5$ faces. Therefore, we obtain that
\begin{eqnarray}\label{fstwo} f(S_2) \ge 5\cdot e(L_2) = 5/2 \cdot v(L_2).
\end{eqnarray}

Our first goal is to show that $L_i\not=\emptyset$ for some $i\ge 3$. If this is false that $L=L_1\cup L_2$ and $l_{17}=v(L_1)+v(L_2).$
Then using (\ref{fsone}) and (\ref{fstwo}) we obtain
\begin{eqnarray}\label{good}
f= f(M) \ge f(S_1) + f(S_2) \ge 5/2\cdot (v(L_1) +v(L_2)) = 5/2\cdot l_{17} \ge 5/2\cdot 4/5 \cdot v(M)= 2v.
\end{eqnarray}
This contradicts the Euler-Poincar\'e theorem for $S^2$ which states that $f=2v-4$. Thus, the graph $L_i$ is nonempty for some $i\ge 3$ and therefore we may find three distinct vertexes $x, y, z$ of $M$ having degree $\le 17$ such that $y$ is connected to $x$ and $z$ by an edge. 

For such a triple $x,y, z$, the link of $y$ in $M$ is homeomorphic to the circle, and we will denote by $d(x, z)$ the distance between $x$ and $z$ in this link. The statement of Theorem \ref{comb} is equivalent to the claim that we may find $x, y, z$ such that additionally to the properties mentioned above one has $d(x, z) \le 2$. 

Suppose the contrary: (A) for any triple of distinct vertexes $x, y, z$ of degree $\le 17$ such that $y$ is connected to both $x$ and $z$ by an edge, the distance $d(x, z)$ (see above) is $\ge 3$. Then $\deg(y)\ge 6$. Denote 
$$R_i= \frac{f(S_i)}{v(L_i)}, \quad i=1, 2, 3, \dots$$
Then $R_1\ge 3$ (by (\ref{fsone})) and $R_2\ge 5/2$ (by (\ref{fstwo})). Let us show that for $i\ge 3$ one has 
\begin{eqnarray}\label{83}
R_i\ge 8/3.
\end{eqnarray} 
Each face of $S_i$, where $i\ge 3$, has one or two vertexes in $L_i$; the possibility that all three vertexes lie in $L_i$ is excluded by our assumption (A). Let $f_1(S_i)$ and $f_2(S_i)$ denote the number of faces in $S_i$ having one or two vertexes in $L_i$, correspondingly. Clearly $f_2(S_i)=2e(L_i)$, since each edge of $L_i$ is incident to exactly two faces. Let $k_i$ denote the number of connected components of $L_i$. 
Then 
\begin{eqnarray}v(L_i)=i\cdot k_i \quad \text{and} \quad e(L_i) \ge v(L_i)-k_i.
\end{eqnarray}
Each vertex of valence 1 of $L_i$ contributes at least one to $f_1(S_i)$; each vertex of valence $\ge 2$ of $L_i$ contributes at least two into $f_1(L_i)$ (due to our assumption (A)). This shows that $f_1(S_i) \ge v(L_i) +k_i$. 
Combining these estimates together we get 
\begin{eqnarray*}
f(S_i) &=& f_1(S_i) + f_2(S_i) \ge v(L_i)+k_i + 2e(L_i)  \\ &\ge& v(L_i) +k_i + 2(v(L_i) -k_i) = 3v(L_i) \left(1-\frac{1}{3i}\right)\ge \frac{8}{3}v(L_i),
\end{eqnarray*}
since $i\ge 3$, proving (\ref{83}). 

Now we have 
\begin{eqnarray}\label{good1}
f=f(M) \ge \sum_{i=1}f(S_i) \ge \frac{5}{2} \sum_{i\ge 1} v(L_i) = \frac{5}{2}l_{17}\ge 2v,
\end{eqnarray}
which contradicts the equality $f=2v-4$. This shows that our assumption (A) is false and completes the proof of Theorem \ref{comb}. 
\end{proof}

Next we state a similar result for triangulations of the torus:

\begin{theorem}\label{combtorus} Let $M$ be a simplicial triangulation of the torus $T^2$. Then there exist three vertexes 
$x, y, z$ of $M$ satisfying $\deg(x)\le 18, \, \deg(y) \le 18, \, \deg(z) \le 18,$
such that either (a) the points $x, y, z$ span a 2-simplex, or (b) for some vertex $w$ the triples 
$x, y, w$ and $y, z, w$ span 2-simplexes of $M$. 
\end{theorem}

\begin{proof} The proof repeats the arguments of Theorem \ref{comb}; we will indicate only the places where changes are needed. 
Firstly, using Lemma \ref{lm1} one obtains the inequality $l_{18} \ge 13/16\cdot v$. Then the string of inequalities (\ref{good}) becomes 
$$f\ge \frac{5}{2}\cdot \frac{13}{16} \cdot v >2v.$$
This contradicts the Euler-Poincar\'e theorem for the torus stating $f=2v$. 
Secondly, instead of (\ref{good1}) we will have
$$f \ge \frac{5}{2} l_{18} \ge \frac{5}{2}\cdot \frac{13}{16}\cdot v >2v,$$
again contradicting the equation $f=2v$. 
\end{proof}

Theorem \ref{combtorus} will not be used in this paper. Similarly to Theorem \ref{comb} it could be applied to study maps of tori into random 2-complexes 
or equivalently to finding pairs of commuting elements in their fundamental groups. 

\section{Proof of Theorem \ref{thmlist}} \label{prfthmlist}

\begin{proof} 
We define $\LL=\LL_1\cup \LL_2$ as the union (\ref{union}) with $\LL_1=\{\Sigma\}$, see above. Our task now is to describe $\LL_2$.

Consider the set $\LL'$ of isomorphism types of all possible simplicial 2-complexes $S$ with the following properties:

\begin{enumerate}
\item  $S$ has at most $47$ two-dimensional simplexes;

\item $S$ can be topologically embedded into the sphere $S^2$;

\item $S$ contains three internal vertices $x, y, z$, each of degree $\le 17$, such that the links of $x, y, z$ are circles and 
$$S=\st(x)\cup \st(y)\cup \st(z);$$

\item The vertexes $x, y, z$ satisfy the conditions of Theorem \ref{comb}, i.e. either the vertices $x, y, z$ span a 2-simplex, or for some vertex $w\in S$ the triples 
$x, y, w$ and $x, z, w$ span 2-simplexes of $S$. 
\end{enumerate}

Property 1 in the description of $\LL'$ obviously follows from Properties 3 and 4; it is stated here to emphasise the finiteness of $\LL'$.


%
%
%

Next we define the set $\LL''$ as the set of all regular quotients of complexes from $\LL'$. Finally, we define $\LL_2$ as the subset of those complexes
of $\LL''$ which contain no tetrahedron as a subcomplex. 

Let us show that the list $\LL$ satisfies condition (1) of Theorem \ref{thmlist}. 
Assume that $Y$ is a finite 2-complex containing no subcomplexes isomorphic to $S\in \LL$. 
Consider all tetrahedra $T_1, \dots, T_k$ contained in $Y$. 
They must be pairwise face disjoint since we know that $Y$ contains no subcomplexes isomorphic to $\Sigma\in \LL_1$. 

Let $Y'\subset Y$ be a subcomplex 
such that $T_j\not\subset Y'$ for $j=1, \dots, k$. We want to show that $Y'$ is aspherical. Assume the contrary, i.e. suppose that $\pi_2(Y') \not=0$.
By Lemma \ref{ronan} there exists a regular simplicial quotient $S'$ of a triangulation of $S^2$ which is embedded into $Y'$. 

By Theorem \ref{comb}, due to our construction of $\LL'$, and because of our assumption that $Y$ has no subcomplexes isomorphic to 
$S\in \LL_2$, the image 
$S'$ must contain a tetrahedron, a contradiction. 

To finish the proof of Theorem \ref{thmlist} we are now left to show that condition (2) is satistied for $S\in \LL_2$, i.e. one has $\mu(S) \le 46/47$. 

Assume that $S$ contains a closed 2-complex $S'$ as a subcomplex\footnote{Recall that a 2-complex is called closed if it has no free edge; an edge is free if it is incident to a single 2-simplex.}. Because of our construction we may assume that $S'$ is not a tetrahedron. 
Then, using Lemma 34 from \cite{CCFK} we obtain $\tilde \mu(S)\le \mu(S) \le 5/6$ which implies that $\tilde \mu(S) \le 46/47$.

Performing repeatedly all possible collapses through free edges of $S$ we may arrive either (a) to a closed 2-complex $S'$ or (b) to a graph. 
In the case (a) $S$ contains a closed subcomplex  and we already know that then $\tilde \mu(S) \le 46/47$. Assuming (b) we have 
 $\chi(S) \le 1$ and, using formula (\ref{muel}), we have 
\begin{eqnarray}\label{this}
\mu(S) = \frac{1}{2} + \frac{2\chi(S)+ L(S)}{2f(S)}\le \frac{f(S)+L(S) +2}{2f(S)}, 
\end{eqnarray}
where $L(S) = \sum_{e}\left(2-\deg_S(e)\right)$. We will show (see below) that for $S\in \LL_2$ one has
\begin{eqnarray}\label{below}
L(S) \le f(S)-3.
\end{eqnarray}
Thus, (\ref{this}) implies that 
$$\mu(S) = \frac{v(S)}{f(S)} \le \frac{f(S)-1/2}{f(S)} $$
implying (since $v(S)$ and $f(S)$ are integers) that $v(S) \le f(S)-1$ and hence 
$$\mu(S) \le \frac{f(S)-1}{f(S)}\le \frac{46}{47}$$
(using $f(S)\le 47$).

Finally we prove the inequality (\ref{below}).  
Denote by $\partial S$ the set of all free edges of $S$. From the definition of $L(S)$ it is clear that $L(S)\le |\partial S|$. 
We claim that for $S\in \LL_2$ one has
\begin{eqnarray}\label{below1}
|\partial S| \le f(S) -3,
\end{eqnarray}
which would imply (\ref{below}). From the construction of $\LL_2$ it is clear that each 2-simplex of $S\in \LL_2$ incident to the boundary $\partial S$ contains exactly one free edge. Hence (\ref{below1}) follows once we know that each $S\in \LL_2$ has at least three 2-simplexes which 
have no free edges.

Given a complex $S'\in \LL'$ and a regular simplicial quotient $f: S'\to S$ where $S\in \LL_2$,
consider Figure \ref{sigmas} where on the left we have three vertices $x, y, z\in S'$ satisfying condition (a) of Theorem \ref{comb} and on the right we have points $x, y, z, w\in S'$ as in condition (b) of Theorem \ref{comb}. 

It is easy to see that in both cases the images under $f$ of the simplexes $\sigma_1, \sigma_2, \sigma_3$ are internal (i.e. have no free edges\footnote{This is an implication of the fact that for an edge $e$ of $S'$ the degree of $e$ in $S'$ is less or equal than the degree of $f(e)$ in $S$, as follows from the regularity condition of Definition \ref{reg}.}) and are all pairwise distinct.
Indeed, $f(\sigma_1) \not=f(\sigma_2)$ and $f(\sigma_1)\not= f(\sigma_3)$ since otherwise the regularity condition of Definition \ref{reg} is violated. 
Besides, $f(\sigma_2)\not=f(\sigma_3)$ as otherwise, in case (b) (see Figure \ref{sigmas} right) one must have $f(x) \not= f(w)$ (to avoid degeneration of $\sigma_1$) and 
hence $f(x)=f(z)$ which produces a folding along the edge $yw$ contradicting Definition \ref{reg}. In the case (a) the arguments are similar. 
\begin{figure}[h]\label{sigmas}
\centering
\includegraphics[width=0.6\textwidth]{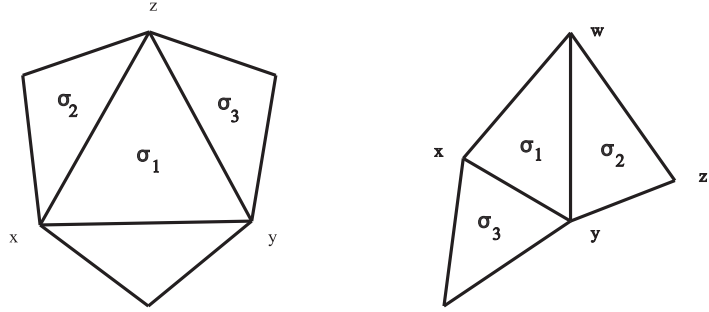}
\caption{Simplicial complexes in the list $\LL_2$.}
\end{figure}

This completes the proof. 

\end{proof}

\section{Proof of Theorem \ref{thm2}}\label{prfthm2}

Suppose that for $Y\in Y(n,p)$ we have the tetrahedra $T_1, \dots, T_k\subset Y$ as in Definition \ref{def1}. Let $Z$ be obtained from $Y$ by removing 
a face from each of the tetrahedra $T_i$. We want to show that $b_2(Z)>0$, a.a.s. Since $\pi_1(Z) = \pi_1(Y)$ and $Z$ is aspherical, it would imply that 
$\cd(\pi_1(Y))=2$. 

Consider two random variables $b_2, k: Y(n,p)\to \Z$, where for $Y\in Y(n,p)$ the symbol $k(Y)$ denotes the number of tetrahedra simplicially embedded in $Y$. Let $f_2: Y(n,p)\to \Z$ denote the number of 2-simplexes in a random complex; this is a binomially distributed random variable with expectation $p\binom n 3$. We have 
\begin{eqnarray}\label{ftwo}
f_2 - \binom {n-1}{2}\, \le\,  b_2\, \le\, f_2 . 
\end{eqnarray}
The inequality (2.6) on page 26 of \cite{JLR} with $t=n^{7/4}$ gives
$$\P(f_2 \le p\binom n 3 -t) \le \exp\left(-\frac{t^2}{2p\binom n 3}\right) \le \exp\left(-\sqrt{n}\right).$$
Thus we see that with probability at least $1- \exp\left(-\sqrt{n}\right)$, one has 
$$b_2\ge f_2 - \binom {n-1}{2}\ge p{\binom n 3}- {\binom {n-1} 2}-t \ge \binom {n-1} 2 \cdot \left(\frac{c-3}{3}\right) -t \ge n^2\cdot \frac{c-3}{7}
 $$
for large $n$. 

The expectation of $k$ satisfies $\E(k) \le p^4n^4< n^{4\epsilon}.$ The Markov inequality $\P(k\ge t) \le \frac{\E(k)}{t}$ with $t=n^{5\epsilon}$ gives
$\P(k\ge n^{5\epsilon}) \le n^{-\epsilon}.$ 
Thus, with probability tending to one as $n\to \infty$, one has
$$b_2(\pi_1(Y)) = b_2(Z) \ge b_2(Y) - k(Y) \ge n^2\cdot \frac{c-3}{7} - n^{5\epsilon} \ge n^2\cdot \frac{c-3}{8}>0$$
for $n\to \infty$. This proves the left inequality in (\ref{secondbetti}).

From inequality (2.5) on page 26 of \cite{JLR} with $t=n^{7/4}$ one obtains that with probability at least $1- \exp(-\sqrt{n})$ one has
$f_2\, \le\,  p\binom n 3 +t $ and hence
$$ b_2(\pi_1(Y)) = b_2(Z) \le b_2(Y) \,  \le \, f_2\, \le\,  p\binom n 3 +t \, \le \, n^{2+\epsilon}.$$
This proves the right inequality in (\ref{secondbetti}).

From Theorem \ref{thm1} we know that for $Y\in Y(n,p)$ one has $\cd(\pi_1(Y)) \le 2$ and from statement (A) we obtain that $\cd(\pi_1(Y))\ge 2$ a.a.s., since the group $\pi_1(Y)$ has nontrivial 2-dimensional homology. 
This implies statement (B).

%

\section*{Acknowledgement} 
The authors are thankful to the referees for their helpful comments.

\bibliographystyle{amsalpha}

\begin{thebibliography}{99}


\bibitem{ALLM} L. \ Aronshtam, N. \ Linial, T. \ {\L}uczak, R. \ Meshulam, \textit{Vanishing of the top homology of a random complex}, arXiv:1010.1400. 


\bibitem{BHK} E.\ Babson, C.\ Hoffman, M.\ Kahle, 
{\it The fundamental group of random $2$-complexes}, J. Amer. Math. Soc. 24 (2011), 1-28. 

\bibitem{B} W.\ A.\ Bogley, \textit{J.H.C. Whitehead's asphericity question}, in: "Two-dimensional Homotopy and Combinatorial Group Theory", eds. 
C. Hog-Angeloni, A. Sieradski and W. Metzler, LMS Lecture Notes 197, Cambridge Univ Press (1993), 309-334. 

\bibitem{Br} K.S. Brown, \textit{Cohomology of groups}, Graduate Texts in Mathematics 87, Springer-Verlag, 1994. 

\bibitem{CCFK} D. Cohen, A.E. Costa, M. Farber, T. Kappeler, \textit{Topology of random 2-complexes,} 
Journal of Discrete and Computational Geometry, {\bf 47}(2012), 117-149.

\bibitem{ER} P.\ Erd\H{o}s, A.\ R\'enyi, {\it On the evolution of 
random graphs}, Publ.\ Math.\ Inst.\ Hungar.\ Acad.\ Sci.\ {\bf 5}
(1960), 17--61.

\bibitem{Hat02} A.\ Hatcher, {\em Algebraic Topology}, Cambridge, 2002.

\bibitem{JLR} S. Janson, T. {\L}uczak, A. Ruci\'nski, \textit{Random graphs}, Wiley-Intersci. Ser. Discrete Math. Optim., Wiley-Interscience, New York, 2000.



\bibitem{Ko} D. Kozlov, \textit{The threshold function for vanishing of the~top homology group of random $d$-complexes},
Proc. Amer. Math. Soc. 138 (2010), 4517-4527.  

\bibitem{LM} N.\ Linial, R.\ Meshulam, {\it Homological connectivity
  of random $2$-complexes}, Combinatorica {\bf 26} (2006),  475--487.

\bibitem{MW} R.\ Meshulam, N.\ Wallach, {\it Homological
  connectivity of random $k$-complexes}, Random Structures \& Algorithms 
  \textbf{34} (2009), 408--417. 



\bibitem{Ron} M.A. Ronan, \textit{On the second homotopy group of certain simplicial complexes and some combinatorial applications}, Quart. J. Math. {\bf {32}}(1981), 225 - 233. 

\bibitem{R} S. Rosenbrock, \textit{The Whitehead Conjecture - an overview}, Siberian Electronoc Mathematical Reports, {\bf 4}(2007), 440-449. 

\end{thebibliography}

%
%
%
%
%
%


\end{document}